\documentclass[11pt]{amsart}

\usepackage[margin=1.15in]{geometry}
\usepackage{amsmath,amssymb,amsthm,mathtools}
\usepackage{microtype}
\usepackage[colorlinks=true,linkcolor=blue,citecolor=blue,urlcolor=blue]{hyperref}

\numberwithin{equation}{section}

\newtheorem{theorem}{Theorem}[section]
\newtheorem{proposition}[theorem]{Proposition}

\newtheorem{corollary}[theorem]{Corollary}
\theoremstyle{remark}
\newtheorem{remark}[theorem]{Remark}

\newcommand{\R}{\mathbb R}
\newcommand{\C}{\mathbb C}
\newcommand{\D}{\mathbb D}
\newcommand{\SH}{\operatorname{SH}}
\newcommand{\PSH}{\operatorname{PSH}}

\title[Perron envelopes with bounded minorants]{Perron envelopes with
\\ globally bounded minorants}

\author{Frank Wikstr\"om}
\address{Centre for Mathematical Sciences, Lund University,
Box 118, SE-221 00 Lund, Sweden}
\email{frank.wikstrom@math.lth.se}


\begin{document}

\begin{abstract}
Let $f$ be an upper semicontinuous function on a domain $\Omega$, and let $V$
be the Perron envelope of $f$ formed only with subharmonic minorants that are
bounded above on all of $\Omega$.  The usual argument for upper
semicontinuity of Perron envelopes breaks down, since the regularization
$V^*$ need not be bounded above; one only knows that $V=V^*$ outside a polar
set.  We show that nevertheless $V=V^*$ everywhere in $\R^d$, using a local
harmonic correction given by a Poisson integral on a ball.  The
plurisubharmonic analogue fails.  We construct an explicit bounded B-regular
Hartogs domain in $\C^n$, $n\geq2$, and a positive pluriharmonic function $W$
whose envelope of bounded-above plurisubharmonic minorants equals $W$ off an
analytic disc and vanishes on it, although $W$ admits a positive
plurisuperharmonic majorant $G$ with $W\leq\varepsilon G+C_\varepsilon$ for
all $\varepsilon>0$.  For pluriharmonic $W$, we show that the envelope
recovers $W$ exactly on the union of the sets where such majorants are
finite.  The question arises in the study of quasibounded plurisubharmonic
functions.
\end{abstract}

\maketitle

\section{Introduction}

Let $\Omega$ be a domain in $\R^d$, $d\geq2$, and let
$f\colon\Omega\to[-\infty,+\infty)$ be upper semicontinuous.  We consider the
Perron envelope formed only with minorants that are bounded above on all
of~$\Omega$:
\begin{equation}\label{eq:real-envelope}
 P_b f(x)=\sup\{u(x):u\in\SH(\Omega),\ u\leq f,
                         \ \sup_\Omega u<+\infty\}.
\end{equation}
For the usual envelope, taken over all subharmonic minorants of $f$, upper
semicontinuity comes for free: its upper semicontinuous regularization is
subharmonic, lies below $f$ because $f$ is upper semicontinuous, and is
therefore itself a competitor.  The global bound breaks this argument, since
$(P_bf)^*$ need not be bounded above when $f$ is not.  What survives is the
following.  Finite maxima make the defining family directed, and $P_b f$ is
locally bounded above because $f$ is, so the Brelot--Cartan theorem and
Choquet's lemma show that $(P_bf)^*$ is subharmonic (or identically $-\infty$)
and
\begin{equation}\label{eq:qe-agreement}
                 P_bf=(P_bf)^* \quad\text{quasi-everywhere}.
\end{equation}
The question is whether equality holds everywhere.  Our first result gives a
positive answer, without any regularity assumption on~$\partial\Omega$.

\begin{theorem}[Subharmonic envelope theorem]\label{thm:subharmonic}
Let $\Omega\subset\R^d$ be a domain ($d\geq2$) and let
$f\colon\Omega\to[-\infty,+\infty)$ be upper semicontinuous.  Then $P_bf$ is
upper semicontinuous, that is, $P_bf=(P_bf)^*$; in particular $P_bf$ is
either subharmonic or identically $-\infty$.  The same holds if the defining
family in \eqref{eq:real-envelope} is restricted to functions which are
bounded on both sides.
\end{theorem}

More generally, the proof shows that $P_bf=(P_bf)^*$ for every upper
semicontinuous $f\colon\Omega\to[-\infty,+\infty]$ for which $(P_bf)^*$ is
subharmonic.  The proof is a local harmonic correction on a ball, given by a
Poisson integral.  It follows that if a subharmonic function $W$ is the
quasi-everywhere limit of an increasing sequence of bounded-above subharmonic
functions, then $W$ is the pointwise supremum of all its bounded-above
subharmonic minorants (Corollary~\ref{cor:qe-minorants}).

For plurisubharmonic functions on a domain $\Omega\subset\C^n$, write
\begin{equation}\label{eq:complex-envelope}
 P_b^{\mathrm{psh}}f(z)
 =\sup\{u(z):u\in\PSH(\Omega),\ u\leq f,
                         \ \sup_\Omega u<+\infty\}.
\end{equation}
Since negligible sets are pluripolar~\cite{BedfordTaylor:82}, the same
argument shows that $(P_b^{\mathrm{psh}}f)^*$ is plurisubharmonic and agrees
with $P_b^{\mathrm{psh}}f$ outside a pluripolar set.  Pointwise equality,
however, fails, even under strong assumptions on both the domain and the
obstacle.  Recall that a bounded domain $\Omega\subset\C^n$ is
\emph{B-regular} if every continuous function on $\partial\Omega$ extends to a
plurisubharmonic function on $\Omega$ which is continuous on
$\overline\Omega$~\cite{Sibony:87}.  These are exactly the domains on which
the Perron--Bremermann envelope of every continuous boundary function is
continuous on $\overline\Omega$ with the prescribed boundary
values~\cite{Sibony:87,Walsh:68}.  The following example shows that this
regularity of the domain does not help once the competitors are required to
be globally bounded.

\begin{theorem}[Plurisubharmonic counterexample]\label{thm:counterexample}
Let $n\geq2$.  There exist a bounded B-regular domain $\Omega\subset\C^n$, an
analytic disc $H\subset\Omega$, and a positive, continuous, unbounded
pluriharmonic function $W$ on $\Omega$ such that $V=P_b^{\mathrm{psh}}W$
satisfies
\[
 V=W\quad\text{on }\Omega\setminus H,
 \qquad V=0<W\quad\text{on }H.
\]
In particular, $V$ is not upper semicontinuous, and $V^*=W$.  The same
conclusion holds if all competitors in \eqref{eq:complex-envelope} are
required to be bounded on both sides.  Moreover, there is a positive
plurisuperharmonic function $G$ on $\Omega$ such that
\[
 W\leq\varepsilon G+C_\varepsilon\quad\text{on }\Omega
 \qquad\text{for every }\varepsilon>0.
\]
\end{theorem}

The domain, the disc and the function are explicit; see \eqref{eq:Omega},
\eqref{eq:W} and Remark~\ref{rem:higher-dim}.  The function $G$ is a strong
majorant of $W$ in the sense of Nilsson~\cite{Nilsson:22}; in particular $W$ is
tame in the sense of~\cite{NilssonWikstrom:21} and pluri-quasibounded in the
sense of~\cite{Nilsson:26} (Remark~\ref{rem:qb}).

Thus $W$ is an increasing limit of bounded plurisubharmonic functions off an
analytic disc, but it cannot be recovered on that disc by taking the supremum
of \emph{all} bounded plurisubharmonic minorants.  The example therefore
separates pointwise bounded approximation from bounded approximation
quasi-everywhere.  The rapid collapse of the Hartogs fibres is what allows
bounded approximation off $H$, while the restriction to $H$ of any bounded
competitor is a bounded-above subharmonic minorant of the Poisson kernel, and
hence nonpositive.  Since plurisubharmonic and subharmonic functions coincide
in $\C=\R^2$, Theorems~\ref{thm:subharmonic} and~\ref{thm:counterexample}
together separate exactly the case $n=1$ from the cases $n\geq2$.

Two further results in Section~\ref{sec:counterexample} describe where the
defect can occur.  Proposition~\ref{prop:jensen} bounds the restricted envelope
in terms of boundary Jensen measures; in the example, every point of $H$ has
such a measure on the boundary circle of $H$ which does not see the pole of
$W$.  Proposition~\ref{prop:strong-majorant} shows that the defect set
$\{V<W\}$ lies in the pole set of every strong majorant of $W$ and, when $W$ is
pluriharmonic, equals the intersection of these pole sets.

The question arose in the study of quasibounded plurisubharmonic
functions~\cite{NilssonWikstrom:21}.  A harmonic function is called
quasibounded if it is the increasing limit of bounded harmonic functions; the
notion goes back to Parreau~\cite{Parreau:51}, see also Arsove and
Leutwiler~\cite{ArsoveLeutwiler:74}.  For the plurisubharmonic analogue
introduced in~\cite{NilssonWikstrom:21}, the natural bounded approximants
converge only outside a pluripolar set, and one is led to ask whether the
supremum of all bounded-above minorants recovers the function everywhere.
Related continuity questions for unbounded Perron--Bremermann envelopes, in
connection with the complex Monge--Amp\`ere equation, were studied
in~\cite{Nilsson:22,Nilsson:26}.  We have not found a treatment of the
pointwise question in the literature, even in the subharmonic case.

Theorem~\ref{thm:subharmonic} is proved in Section~\ref{sec:subharmonic}.
Section~\ref{sec:counterexample} contains the Hartogs domain and its
B-regularity, the proof of Theorem~\ref{thm:counterexample}, and
Propositions~\ref{prop:jensen} and~\ref{prop:strong-majorant}.

\section{The subharmonic envelope}\label{sec:subharmonic}

We use $u^*$ for upper semicontinuous regularization.  Polar sets and
quasi-everywhere statements are understood in the classical sense.  We use
only standard facts about subharmonic functions; general references are
\cite{ArmitageGardiner:01,Ransford:95}.

\begin{proof}[Proof of Theorem~\ref{thm:subharmonic}]
Put $V=P_bf$ and $W=V^*$, where $V$ is formed with either of the two
defining families.  If $f<+\infty$, then $f$, and hence the defining family,
is locally bounded above, so $W$ is subharmonic or identically $-\infty$ by
the Brelot--Cartan theorem.  We prove $V=W$ under this conclusion alone,
which also gives the more general statement after the theorem.

If $W\equiv-\infty$, the conclusion is immediate. Hence assume
$W\not\equiv-\infty$.  Since $V\leq f$ and $f$ is upper semicontinuous,
\begin{equation}\label{eq:W-below-f}
                             W\leq f.
\end{equation}
By Choquet's lemma, followed by taking finite maxima, there is an increasing
sequence $(u_j)$ in the defining family such that
\begin{equation}\label{eq:choquet-sequence}
                    u_j\nearrow W\quad\text{quasi-everywhere}.
\end{equation}
Since $W\not\equiv-\infty$, discarding finitely many terms we may assume that
no $u_j$ is identically $-\infty$.

Fix $p\in\Omega$ with $W(p)>-\infty$, choose a ball $B=B(p,r)\Subset\Omega$,
and let $\sigma$ be normalized surface measure on $\partial B$.  For
$\varphi\in L^1(\sigma)$ we write $P_B[\varphi]$ for the Poisson integral of
$\varphi$; it is harmonic on $B$, and $P_B[\varphi](p)=\int\varphi\,d\sigma$
since $B$ is centred at~$p$.  We use two standard facts about Poisson
integrals \cite{ArmitageGardiner:01}.  First, if $\varphi$ is upper
semicontinuous on $\partial B$, then
\begin{equation}\label{eq:poisson-usc}
 \limsup_{B\ni x\to\xi}P_B[\varphi](x)\leq\varphi(\xi)
 \qquad(\xi\in\partial B),
\end{equation}
as one sees by approximating $\varphi$ from above by a decreasing sequence of
continuous functions.  Second, if $s$ is subharmonic on a neighbourhood of
$\overline B$, then $s\leq P_B[s]$ on $B$.

The functions $W,u_1,u_2,\ldots$ belong to $L^1(\sigma)$.  Indeed, they are
bounded above on $\partial B$, and if one of them, say $s$, had
$\int s\,d\sigma=-\infty$, then $s\leq P_B[s]\equiv-\infty$ on $B$, which is
impossible for a subharmonic function $s\not\equiv-\infty$ on the domain
$\Omega$.  Moreover, $\sigma$ is harmonic measure for $B$ at $p$, and harmonic
measure does not charge polar sets.  Hence \eqref{eq:choquet-sequence} holds
$\sigma$-almost everywhere on $\partial B$.

Recall that $u_j\leq V\leq W$.  Define the harmonic correction
\[
 h_j=P_B[u_j]-P_B[W]=P_B[u_j-W]
\]
and put
\begin{equation}\label{eq:correction}
 s_j=W+h_j=\bigl(W-P_B[W]\bigr)+P_B[u_j]\qquad\text{on }B.
\end{equation}
Since $h_j$ is harmonic and $h_j\leq0$, the function $s_j$ is subharmonic on
$B$ and satisfies $s_j\leq W$.  Let $\xi\in\partial B$.  The first term in
\eqref{eq:correction} is nonpositive, and $u_j|_{\partial B}$ is upper
semicontinuous, so \eqref{eq:poisson-usc} gives
\[
 \limsup_{B\ni x\to\xi}s_j(x)\leq\limsup_{B\ni x\to\xi}P_B[u_j](x)
 \leq u_j(\xi).
\]
Here no expression of the form $\infty-\infty$ occurs, since $P_B[W]$ is
finite on $B$ and $P_B[u_j]$ is bounded above.

The pasting lemma now shows that
\begin{equation}\label{eq:pasted-function}
 \widetilde u_j=
 \begin{cases}
   \max\{u_j,s_j\},&\text{on }B,\\
   u_j,&\text{on }\Omega\setminus B,
 \end{cases}
\end{equation}
is subharmonic on $\Omega$.  Moreover,
\[
                     \widetilde u_j\leq W\leq f.
\]
It is globally bounded above: outside $B$ it equals $u_j$, while on $B$ it
is bounded above by the upper bound of $W$ on $\overline B$.  If $u_j$ is
bounded below, then $\widetilde u_j\geq u_j$, so two-sided boundedness is
also preserved.  Thus $\widetilde u_j$ is an admissible minorant in either
version of the theorem.

Finally, $u_j-W$ is defined $\sigma$-almost everywhere, increases to $0$
$\sigma$-almost everywhere, and $u_1-W\in L^1(\sigma)$.  By monotone
convergence,
\[
 \widetilde u_j(p)\geq s_j(p)
 =W(p)+\int_{\partial B}(u_j-W)\,d\sigma
 \longrightarrow W(p).
\]
It follows that $V(p)\geq W(p)$.  The opposite inequality is automatic.
Points at which $W=-\infty$ are trivial, and hence $V=W$ throughout
$\Omega$.
\end{proof}

\begin{corollary}\label{cor:qe-minorants}
Let $W\in\SH(\Omega)$.  Suppose that there is an increasing sequence of
globally bounded-above subharmonic functions $u_j\leq W$ such that
$u_j\to W$ quasi-everywhere.  Then
\[
 W(x)=\sup\{u(x):u\in\SH(\Omega),\ u\leq W,
                         \ \sup_\Omega u<+\infty\}
 \qquad(x\in\Omega).
\]
The same holds with two-sided boundedness throughout.
\end{corollary}

\begin{proof}
The upper semicontinuous regularization of the envelope on the right is
$W$, so Theorem~\ref{thm:subharmonic} applies.
\end{proof}

\begin{remark}\label{rem:linearity}
The harmonic correction is the only place where linearity is used.  The sum
$W+h_j$ remains subharmonic, whereas adding a merely subharmonic correction
to a plurisubharmonic function need not preserve plurisubharmonicity.  The
counterexample below shows that this is not only a limitation of the proof.
\end{remark}

\section{A B-regular pluripotential counterexample}\label{sec:counterexample}

Let
\begin{equation}\label{eq:poisson-kernel}
 P(w)=\frac{1-|w|^2}{|1-w|^2}
     =\operatorname{Re}\frac{1+w}{1-w},
 \qquad w\in\D,
\end{equation}
be the Poisson kernel of the unit disc with pole at $1$.  Define
\begin{equation}\label{eq:Phi}
                 \Phi(w)=P(w)^2+\frac1{1-|w|^2}.
\end{equation}
The function $\Phi$ is smooth and strictly subharmonic on $\D$, and
$\Phi(w)\to+\infty$ as $w\to\partial\D$.

Consider the complete Hartogs domain
\begin{equation}\label{eq:Omega}
 \Omega=\{(z,w)\in\C\times\D:|z|<e^{-\Phi(w)}\}.
\end{equation}
It is bounded and pseudoconvex: away from $z=0$ it is given by the
plurisubharmonic inequality
\[
                         \log|z|+\Phi(w)<0.
\]
We first verify the claimed boundary regularity.

\begin{proposition}\label{prop:B-regular}
The domain $\Omega$ in \eqref{eq:Omega} is B-regular.
\end{proposition}

\begin{proof}
Since $\Phi$ tends to $+\infty$ at every point of $\partial\D$, the boundary
is the disjoint union
\begin{align*}
 S&=\{(z,w):w\in\D,\ |z|=e^{-\Phi(w)}\},\\
 T&=\{(0,\zeta):|\zeta|=1\}.
\end{align*}
Every point of $S$ is a smooth strictly pseudoconvex boundary point.  In
fact,
\[
                         r(z,w)=\log|z|+\Phi(w)
\]
is a smooth defining function there and $dd^cr=dd^c\Phi(w)$.  Every nonzero
complex tangent vector has a nonzero $w$-component, so the Levi form is
positive on the complex tangent space.  A local strong plurisubharmonic
barrier at such a point can be globalized by taking its maximum with a
negative constant outside a smaller neighbourhood.

For a cuspidal point $(0,\zeta)\in T$, the function
\begin{equation}\label{eq:tip-barrier}
                 b_\zeta(z,w)=\operatorname{Re}(\overline\zeta w)-1
\end{equation}
is a global strong plurisubharmonic barrier.  Indeed, it is continuous and
pluriharmonic, vanishes at $(0,\zeta)$, and is strictly negative at every
other point of $\overline\Omega$.  The last assertion uses the collapse of
the fibres: $(0,\zeta)$ is the only point of $\overline\Omega$ whose second
coordinate is $\zeta$.

Thus every boundary point admits a strong plurisubharmonic barrier.
Sibony's characterization of B-regularity~\cite{Sibony:87} proves the result.
\end{proof}

We now calculate the envelope of the positive pluriharmonic function
\begin{equation}\label{eq:W}
                            W(z,w)=P(w)
\end{equation}
on $\Omega$.

\begin{proof}[Proof of Theorem~\ref{thm:counterexample}]
We give the proof for $n=2$; the changes needed for $n>2$ are described in
Remark~\ref{rem:higher-dim}.  Let
\[
                         H=\{0\}\times\D.
\]
If $u\in\PSH(\Omega)$ is bounded above and $u\leq W$, then its restriction
\[
                         v(w)=u(0,w)
\]
is either identically $-\infty$ or is a bounded-above subharmonic function
on $\D$ satisfying $v\leq P$.  For every
$\zeta\in\partial\D\setminus\{1\}$,
\[
             \limsup_{w\to\zeta}v(w)\leq0.
\]
The exceptional singleton has harmonic measure zero.  The extended maximum
principle, or equivalently resolutivity of the boundary datum which is zero
except at one point, gives $v\leq0$ on $\D$; compare~\cite[Theorem
3.6.9]{Ransford:95}.  Since the constant zero function is an admissible
minorant,
\begin{equation}\label{eq:V-on-H}
                              V(0,w)=0.
\end{equation}

For $\varepsilon>0$, define
\begin{equation}\label{eq:uepsilon}
 u_\varepsilon(z,w)
 =\max\{0,P(w)+\varepsilon\log|z|\},
\end{equation}
where $u_\varepsilon(0,w)=0$.  This function is plurisubharmonic on
$\Omega$.  Since $|z|<1$, it satisfies $u_\varepsilon\leq P=W$.  Moreover,
the definition of $\Omega$ gives
\[
                    \log|z|<-\Phi(w)\leq-P(w)^2.
\]
Consequently,
\begin{equation}\label{eq:uepsilon-bound}
 0\leq u_\varepsilon(z,w)
 \leq\max\{0,P(w)-\varepsilon P(w)^2\}
 \leq\frac1{4\varepsilon}.
\end{equation}
Thus $u_\varepsilon$ is bounded on both sides and belongs to the defining
family for $V$.  At every point with $z\neq0$,
\begin{equation}\label{eq:uepsilon-limit}
                 u_\varepsilon(z,w)\longrightarrow P(w)
                 \qquad(\varepsilon\downarrow0).
\end{equation}
Together with $V\leq W$, equations \eqref{eq:V-on-H} and
\eqref{eq:uepsilon-limit} prove
\begin{equation}\label{eq:V-formula}
 V(z,w)=
 \begin{cases}
   P(w),&z\neq0,\\
   0,&z=0.
 \end{cases}
\end{equation}
Since $H$ has empty interior and $W$ is continuous, it follows that
$V^*=W$.  In particular,
\[
                         V(0,0)=0<1=V^*(0,0).
\]
The two-sided version follows since each $u_\varepsilon$ is bounded on both
sides.  The majorant $G$ is given in Remark~\ref{rem:qb}, which completes the
proof.
\end{proof}

The geometry behind the example is best seen through Jensen measures.  Let
$\Omega\subset\C^n$ be a bounded domain and $p\in\Omega$.  For
$u\colon\Omega\to[-\infty,+\infty)$ bounded above, let $u^*$ also denote its
upper semicontinuous extension to $\overline\Omega$,
$u^*(\xi)=\limsup_{\Omega\ni x\to\xi}u(x)$ for $\xi\in\partial\Omega$.
Following~\cite{Wikstrom:01}, a probability measure $\mu$ on
$\overline\Omega$ is a \emph{Jensen measure} for $p$ (with respect to
bounded-above functions), written $\mu\in J_p^b$, if
\[
                  u(p)\leq\int_{\overline\Omega}u^*\,d\mu
\]
for every $u\in\PSH(\Omega)$ which is bounded above.  The class $J_p^c$ is
defined in the same way using $u\in\PSH(\Omega)\cap C(\overline\Omega)$.  We
call $\mu\in J_p^b$ a \emph{boundary Jensen measure} if
$\operatorname{supp}\mu\subset\partial\Omega$.

\begin{proposition}[Jensen measures and the restricted envelope]
\label{prop:jensen}
Let $\Omega\subset\C^n$ be a bounded domain, let $f$ be upper semicontinuous
on $\Omega$, and let $V=P_b^{\mathrm{psh}}f$.  For $M\in\R$ write
$f\wedge M=\min\{f,M\}$.  Then, for every $p\in\Omega$,
\begin{equation}\label{eq:jensen-bound}
 V(p)\leq\lim_{M\to+\infty}\,\inf_{\mu}\int_{\partial\Omega}(f\wedge M)^*\,d\mu,
\end{equation}
where the infimum is taken over all boundary Jensen measures $\mu$ for $p$.

In the situation of Theorem~\ref{thm:counterexample}, let $p=(0,w_0)\in H$
and let $\mu_{w_0}$ be the image of harmonic measure for $\D$ at $w_0$ under
$\zeta\mapsto(0,\zeta)$.  Then $\mu_{w_0}$ is a boundary Jensen measure for
$p$ carried by $T=\{(0,\zeta):|\zeta|=1\}$, and
\[
                 \int_T(W\wedge M)^*\,d\mu_{w_0}=0<W(p)
                 \qquad\text{for every }M.
\]
\end{proposition}

\begin{proof}
The integrals in \eqref{eq:jensen-bound} increase with $M$, so the limit
exists.  Let $u$ be a competitor for $V$ and let $M\geq\sup_\Omega u$.  Then
$u\leq f\wedge M$ on $\Omega$, hence $u^*\leq(f\wedge M)^*$ on $\partial\Omega$,
and $u(p)\leq\int u^*\,d\mu\leq\int(f\wedge M)^*\,d\mu$ for every boundary
Jensen measure $\mu$ for $p$.  Taking the infimum over $\mu$, letting
$M\to+\infty$, and then taking the supremum over $u$ gives
\eqref{eq:jensen-bound}.

For the example, let $u\in\PSH(\Omega)$ be bounded above and put
$v(w)=u(0,w)$.  If $v\not\equiv-\infty$, then $v$ is subharmonic and bounded
above on $\D$, and $v^*(\zeta)=\limsup_{w\to\zeta}v(w)\leq u^*(0,\zeta)$ for
$|\zeta|=1$.  Approximating the upper semicontinuous function $v^*$ from
above by continuous functions and applying the maximum principle shows that
$v$ is majorized by the Poisson integral of $v^*$.  Hence
\[
 u(p)=v(w_0)\leq\int_{\partial\D}v^*\,d\omega_{w_0}
      \leq\int_T u^*\,d\mu_{w_0},
\]
where $\omega_{w_0}$ is harmonic measure for $\D$ at $w_0$; the case
$v\equiv-\infty$ is trivial.  Thus $\mu_{w_0}$ is a boundary Jensen measure
for~$p$.  Finally, points of $\Omega$ close to $(0,\zeta)$ have second
coordinate close to $\zeta$, and $P(w)\to0$ as $w\to\zeta\neq1$.  Hence
$(W\wedge M)^*$ vanishes on $T\setminus\{(0,1)\}$, and the remaining point
has $\mu_{w_0}$-measure zero.
\end{proof}

In particular, \eqref{eq:jensen-bound} recovers \eqref{eq:V-on-H}, and a single
measure detects the loss for every truncation: every point of the exceptional
disc $H$ has a boundary Jensen measure carried by the collapsed circle $T$,
and the pole of $W$ at the single point $(0,1)$ carries no mass.  Since
$\Omega$ is B-regular, $J_p^b=J_p^c$ by~\cite{Wikstrom:01}.  Hence the
infimum in \eqref{eq:jensen-bound} may equivalently be taken over boundary
measures in $J_p^c$, the class that governs the Dirichlet problem with
continuous boundary data.

\begin{remark}[Tameness and quasiboundedness quasi-everywhere]\label{rem:qb}
Taking $\varepsilon=1/j$ in \eqref{eq:uepsilon} produces an increasing
sequence of bounded plurisubharmonic functions which converges to $W$ on
$\Omega\setminus H$ and to zero on $H$.  Thus $W$ is quasibounded
quasi-everywhere in the direct approximation sense, while
\eqref{eq:V-formula} shows that no family of bounded plurisubharmonic
minorants can recover it pointwise on $H$.  This is precisely the distinction
which disappears in the subharmonic envelope theorem.

The example remains within the tame class.  Indeed,
\[
                         G(z,w)=-\log|z|
\]
is a positive plurisuperharmonic function on $\Omega$, with value $+\infty$
on $H$, and \eqref{eq:Omega} gives
\[
                  G(z,w)>\Phi(w)\geq P(w)^2=W(z,w)^2.
\]
Consequently $W\leq\varepsilon W^2+1/(4\varepsilon)\leq\varepsilon G+1/(4\varepsilon)$
for every $\varepsilon>0$, which is the last assertion of
Theorem~\ref{thm:counterexample}.  Since $W^2$ has a plurisuperharmonic
majorant, \cite[Theorem
3.2]{NilssonWikstrom:21} implies that $W$ is tame.  In fact, $G/W\geq W$
wherever $z\neq0$; hence $G/W\to+\infty$ as $W\to+\infty$. Equivalently, $G$ is
a nontrivial strong majorant in the sense of~\cite[Lemma 2.2]{Nilsson:22}, and
it strongly quasibounds $W$ in the terminology of~\cite[Definition
3.1]{Nilsson:26}.
\end{remark}

The following observation is the direct restricted-envelope version of the
strong-majorant argument in~\cite[Proposition 2.3]{Nilsson:22}.

\begin{proposition}[A pointwise strong-majorant criterion]
\label{prop:strong-majorant}
Let $\Omega\subset\C^n$ be a domain, let $W\in\PSH(\Omega)$, and put
\[
 V=\sup\{u:u\in\PSH(\Omega),\ u\leq W,\ \sup_\Omega u<+\infty\}.
\]
Call a positive plurisuperharmonic function $G$ on $\Omega$ a \emph{strong
majorant} of $W$ if for every $\varepsilon>0$ there is a constant
$M_\varepsilon$ such that
\[
                         W\leq\varepsilon G+M_\varepsilon.
\]
\begin{enumerate}
\item[(a)] If $G$ is a strong majorant of $W$, then $V=W$ on
$\{G<+\infty\}$.  In particular, a finite-valued strong majorant forces
$V=W$ everywhere.
\item[(b)] If $W$ is pluriharmonic, then conversely $V(p)=W(p)$ if and only
if some strong majorant $G$ of $W$ satisfies $G(p)<+\infty$.  Thus
\[
 \{V=W\}=\bigcup_G\,\{G<+\infty\},
\]
the union being taken over all strong majorants of $W$.
\end{enumerate}
\end{proposition}

\begin{proof}
(a) For $\varepsilon>0$, the function
\[
                         W_\varepsilon=W-\varepsilon G
\]
is plurisubharmonic, is bounded above by $M_\varepsilon$, and satisfies
$W_\varepsilon\leq W$.  Hence $W_\varepsilon\leq V$.  If $G(p)<+\infty$,
then $W_\varepsilon(p)\to W(p)$ as $\varepsilon\downarrow0$, and therefore
$V(p)=W(p)$.

(b) It remains to prove the forward implication.  Suppose that $V(p)=W(p)$.
Since the defining family of $V$ is closed under finite maxima and $W(p)$ is
finite, there is an increasing sequence $(u_k)$ in that family with
\[
                         u_k(p)>W(p)-2^{-k}.
\]
Let $M_k=\sup_\Omega u_k<+\infty$ and define
\[
                         G=1+\sum_{k=1}^\infty(W-u_k).
\]
Each term $W-u_k$ is nonnegative and plurisuperharmonic because $W$ is
pluriharmonic.  Hence $G$ is an increasing limit of positive
plurisuperharmonic functions, and $G(p)<2$; in particular $G\not\equiv+\infty$,
so $G$ is a positive plurisuperharmonic function on~$\Omega$.  For every
$N\geq1$,
\[
 NW=\sum_{k=1}^N u_k+\sum_{k=1}^N(W-u_k)\leq\sum_{k=1}^N M_k+G
\]
at every point where $u_1,\ldots,u_N$ are finite, and trivially elsewhere,
since $G=+\infty$ there.  Given $\varepsilon>0$, choose $N\geq1/\varepsilon$.
Then $W\leq\varepsilon G+N^{-1}\sum_{k\leq N}M_k$, because $G\geq0$.  Thus $G$
is a strong majorant of $W$ with $G(p)<+\infty$.
\end{proof}

For the strong majorant $G=-\log|z|$ in the counterexample,
$\{G=+\infty\}=H$.  By \eqref{eq:V-formula} and
Proposition~\ref{prop:strong-majorant}(a), the defect set $H=\{V<W\}$ is
contained in the pole set of \emph{every} strong majorant of $W$.  By part~(b)
it is the intersection of all these pole sets, and here a single strong
majorant already attains it.

\begin{remark}[The continuous Dirichlet problem]\label{rem:dirichlet}
Although the obstacle in the counterexample is tame, the example does not
contradict the possible existence of globally continuous solutions to the
Dirichlet problem for tame maximal plurisubharmonic functions.  Indeed, the
upper semicontinuous regularization of the restricted envelope is
\[
                             V^*=W=P(w),
\]
which is continuous and pluriharmonic throughout $\Omega$.  Only the raw
pointwise envelope $V$ is discontinuous, while $W$ itself is already a
globally continuous maximal solution.  Notice the subtlety that the strong
majorant $G=-\log|z|$ is identically $+\infty$ on $H$, and by
Proposition~\ref{prop:strong-majorant}(a) so is every strong majorant of $W$.
Its restriction
therefore gives no finite superharmonic majorant of the Poisson kernel on the
disc.  Tameness need not be inherited by an analytic subvariety contained in
the pole set of a strong majorant.  The example consequently answers the
pointwise restricted-envelope question negatively under tameness, but not
the distinct Dirichlet-continuity question.
\end{remark}

\begin{remark}[Higher dimensions]\label{rem:higher-dim}
Let $n\geq2$, write points of $\C^n$ as $(z,w)$ with $z\in\C^{n-1}$ and
$w\in\D$, and let $|z|$ be the Euclidean norm.  Define
\[
 \Omega_n=\{(z,w)\in\C^{n-1}\times\D:|z|<e^{-\Phi(w)}\},
 \qquad W(z,w)=P(w),\qquad H=\{0\}\times\D.
\]
Then Theorem~\ref{thm:counterexample} and everything in this section hold
for $\Omega_n$ with the same proofs, since $\log|z|$ is plurisubharmonic on
$\C^{n-1}$.  The only point needing comment is strict pseudoconvexity along
$S$.  With $r=\log|z|+\Phi(w)$ and a tangent vector $(a,b)\in\C^{n-1}\times\C$,
the Levi form of $r$ is
\[
 \frac{|z|^2|a|^2-|\langle a,z\rangle|^2}{2|z|^4}
 +\frac{\partial^2\Phi}{\partial w\,\partial\bar w}(w)\,|b|^2,
\]
and $(a,b)$ is complex tangent to $S$ if and only if
$\langle a,z\rangle/(2|z|^2)+(\partial\Phi/\partial w)(w)\,b=0$.  If
$b\neq0$, the second term is positive.  If $b=0$ and $a\neq0$, then
$a\perp z$ and the first term equals $|a|^2/(2|z|^2)>0$.  Hence the Levi form
is positive definite on the complex tangent space.  Note that products such
as $\Omega\times\D^{n-2}$ would not do, since their boundaries contain
analytic discs and so they are not B-regular.
\end{remark}

\section*{Acknowledgements}

OpenAI Codex (GPT-5.6 Sol) and Anthropic Claude (Opus~5.5) were used as
interactive research assistants for mathematical brainstorming and proof
auditing.  The author reviewed all generated material, independently checked the
arguments and computations, and takes full responsibility for the content of the
paper.

\end{document}